\documentclass[12pt, reqno]{amsart}

\usepackage{amsfonts,amsmath,amsthm,amssymb}
\usepackage{latexsym}

\usepackage{color}

\usepackage[normalem]{ulem}

\bibliography{file 1, file 2, ....} 
\bibliography{file 1, file 2, ....} 

\def\ns{\footnotesize \it}

\begin{document}

{\theoremstyle{plain}
  \newtheorem{theorem}{Theorem}[section]
  \newtheorem{corollary}[theorem]{Corollary}
  \newtheorem{proposition}[theorem]{Proposition}
  \newtheorem{lemma}[theorem]{Lemma}
  \newtheorem{question}[theorem]{Question}
  \newtheorem{conjecture}[theorem]{Conjecture}
}

{\theoremstyle{definition}
  \newtheorem{definition}[theorem]{Definition}
  \newtheorem{notation}[theorem]{Notation}
  \newtheorem{remark}[theorem]{Remark}
  \newtheorem{example}[theorem]{Example}
}

\title {Gorenstein Hilbert Coefficients}
\maketitle

\author{\begin{center} Sabine El Khoury \end{center}
{\ns Department of Mathematics, American University of Beirut, Beirut, Lebanon. se24@aub.edu.lb} \\  
\begin{center}Hema Srinivasan\end{center}
{ \ns Department of Mathematics, University of Missouri, Columbia, Missouri,
USA. hema@math.missouri.edu
}\\ [.1in]}
\begin{abstract}  We prove upper and lower bounds for all the coefficients in the Hilbert Polynomial of a 
graded Gorenstein algebra $S=R/I$ with a quasi-pure resolution over $R$.  The bounds are in terms of the minimal and the maximal shifts in the resolution of $R$ .  These bounds are analogous to the bounds for the multiplicity found in  \cite{S}  and are stronger than 
the bounds for the Cohen Macaulay algebras found in \cite{HZ}.
\end{abstract}
\section{Introduction}
Let $S=\bigoplus S_i$ be a standard graded $k$-algebra of dimension $d$, finitely generated in degree one.
 $H(S,i) = \dim_kS_i$ is the Hilbert function of $S$.  It is well known that  $H(S,i)$, for $i>>0$,  is a polynomial $P_S(x)$, called the   {\it the Hilbert polynomial} of $S$.  $P_S(x)$ has  degree $d-1$.  If we write,
  $$P_S(x)= \displaystyle \sum_{i=0}^{d-1} (-1)^{i}e_{i}  {  x+d-1-i \choose x} \\
= \frac {e_0}{(d-1)!} x^{d-1}+\ldots +(-1)^{d-1}e_{d-1}\\
$$

\noindent Then the coefficients $e_{i}$ are called the {\it Hilbert coefficients} of $S$.  The first one, $e_0$  called the multiplicity  is the most studied and is denoted by $e$. 

If we write $S = R/I$, where $R$ is the polynomial ring in $n$ variables and $I$ is a homogeneous ideal of $R$, then all these coefficients can be computed from the shifts in the minimal homogenous $R$- resolution 
$\bf{F}$ of $S$ given as follows:
 $$0\rightarrow \displaystyle\bigoplus_{j=m_s}^{M_s} R(-j)^{\beta_{sj} }\stackrel{\delta_s}{\rightarrow} \ldots\rightarrow \displaystyle\bigoplus_{j=m_i}^{M_i}R(-j)^{\beta_{ij}}  \stackrel{\delta_i}{\rightarrow}  \ldots\rightarrow \displaystyle\bigoplus_{j=m_1}^{M_1}R(-j)^{\beta _{1j}}\stackrel{\delta_1} \rightarrow R \rightarrow R/I \rightarrow 0 $$
Let $h =$ height of $I$, so that $h \leq s$.   
In 1995, Herzog and Srinivasan \cite {HS} proved that if this resolution is quasi-pure, i.e. if $m_i\ge M_{i-1}$,  then 
 $$ \frac{\prod_{i=1}^sm_i}{s!}\leq e(S) \leq \frac{ \prod_{i=1}^sM_i}{s!},  \text {if } h=s$$ and   
 $e(S) \leq \frac{\prod_{i=1}^hM_i}{h!}$ if $h<s$.     

Further, Herzog, Huneke and Srinivasan  conjectured  this to hold for all homogeneous algebras $S$ which came to be known as the multiplicity conjecture. 

When $S$ is Gorenstein, Srinivasan established stronger bounds for the multiplicity.  

{\bf Theorem}  [Srinivasan \cite {S}]
{\it If $S$ is a homogeneous Gorenstein algebra with quasi-pure resolution of length $s=2k$ or $2k+1$, then
$$ \frac{m_1 \ldots m_kM_{k+1} \ldots M_s}{s!} \leq e(S) \leq \frac {M_1 \ldots M_k m_{k+1} \ldots m_s}{s!}.$$}
 In this paper, we establish bounds for all the remaining Hilbert coefficients of Gorenstein Algebras with quasi-pure resolutions analogous to the above bounds for the multiplicity. 
We prove in \ref {main}

{\bf {Theorem}}  \ref{main}  {\it  If $S$ is a homogeneous Gorenstein Algebra with quasi-pure resolution of length $s = 2k$ or $2k+1$.  Then, for $0\leq l \leq n-s$, \\

\noindent \small  $f_{l}(m_1 \ldots  m_kM_{k+1} \ldots  M_s) \frac{m_1 \ldots  m_kM_{k+1} \ldots  M_s }{(s+l)!}\leq e_l(S) 
\leq f_{l}(M_1 \ldots  M_km_{k+1} \ldots  m_s)\frac{M_1 \ldots  M_km_{k+1} \ldots  m_s}{(s+l)!}$
\normalsize with $ f_{l}(a_1, \ldots a_s)= \displaystyle \sum_{1 \leq i_1 \leq \ldots  i_l \leq s} \displaystyle \prod_{t=1}^l (a_{i_t}-(i_{t}+t-1))$ and $f_0 = 1$  }
 
Boij and S\"oderberg \cite{BS08}  conjectured that Betti sequences of all graded algebras can be written (uniquely) as sums of positive rational multiples of betti sequences of pure algebras which in turn  implied the multiplicity conjecture. In 2008, these conjectures were proved by Eisenbud and Schreyer \cite{ES} for C-M modules in characteristic zero and extended to non C-M modules by Boij and S\"oderberg \cite{BS08}.
 
Using these results,  Herzog and Zheng \cite{HZ} showed  that if $S$ is Cohen-Macaulay of codimension $s$, then all Hilbert coefficients satisfy
$$\frac{m_1m_2\ldots m_s}{(s+i)!}h_i(m_1, \ldots m_s)\leq e_i(S) \leq \frac{M_1M_2\ldots M_s}{(s+i)!}h_i(M_1, \ldots M_s)$$ 
with $h_i(d_1, \ldots d_s)= \displaystyle \sum_{1 \leq j_1\leq  \ldots  j_i \leq s} \displaystyle \prod_{k=1}^i(d_{j_k}-(j_k+k-1))$  and $h_0(d_1, \ldots d_{s})=1$

Our results extend those of Srinivasan \cite{S} as well as the above result \cite{HZ} to all coefficients of Gorenstein algebras with quasi-pure resolutions.

In section $3$ we give an explicit formula of the Hilbert coefficients as a function of the shifts of the minimal resolution of a Gorenstein algebra.
These expressions depend on whether the projective dimension is even or odd. 

In section $4$, we establish the stronger bounds for the higher Hilbert coefficients 
when the algbera has a quasi-pure resolution.   

\section{Preliminaries and Notations}
Let $R=K [x_1, \ldots x_n]$ be the poynomial ring in $n$ variables,  $I$ be a homogeneous ideal contained in $(x_1,x_2, \ldots , x_n)$ and $S = R/I$.  Let  $\bf{F}$ be the minimal homogeneous resolution of $S$ over $R$ given by:
$$0\rightarrow \displaystyle\bigoplus_{j=1}^{b_s} R(-d_{sj}) \stackrel{\delta_s}{\rightarrow} \ldots\rightarrow \displaystyle\bigoplus_{j=1}^{b_i}R(-d_{ij})  \stackrel{\delta_i}{\rightarrow}  \ldots\rightarrow \displaystyle\bigoplus_{j=1}^{b_1}R(-d_{1j})\stackrel{\delta_1} \rightarrow R \rightarrow R/I \rightarrow 0 $$

\begin{definition} A  resolution is called quasi-pure if $d_{ij} \geq d_{i-1,l}$ for all $j$ and $l$, that is, if $m_i \geq M_{i-1}$ for all $i$.
\end{definition}

Suppose $S$ is Gorenstein. Then by duality of the resolution, the resolution of $S$ can be written as follows.

If $I$ is of height $ 2k+1$ then

$0 \rightarrow  R(-c) {\rightarrow} \displaystyle\sum_{j=1}^{b_1}R(-(c-a_{1j}) \rightarrow \ldots  \rightarrow \displaystyle\sum_{j=1}^{b_k}R(-(c-a_{kj}))$
\begin{equation}  \label{eq1}\hspace {3cm}
 \rightarrow \displaystyle\sum_{j=1}^{b_k}R(-a_{kj})  \rightarrow  \ldots \rightarrow \displaystyle\sum_{j=1}^{b_1}R(-a_{1j}) \rightarrow R 
\end{equation}	

and if $I$ is of height $2k$ then

$0\rightarrow  R(-c) {\rightarrow} \displaystyle\sum_{j=1}^{b_1}R(-(c-a_{1j}) \rightarrow \ldots \rightarrow \displaystyle\sum_{j=1}^{b_k/2=r_k}R(-(c-a_{kj})) \oplus \displaystyle\sum_{j=1}^{b_k/2=r_k}R(-a_{kj}) $
\begin{equation}\label{eq2} \hspace {5cm} \rightarrow  \ldots \rightarrow \displaystyle\sum_{j=1}^{b_1}R(-a_{1j}) \rightarrow R \end{equation}\\
\\
\begin{remark} \begin{enumerate} \item The minimal shifts in the resolution are:

\begin{align*}
m_i  &=  min_j a_{ij} & \hspace {0.25cm} 1\leq i \leq k \\
&= c - max_ja_{s-i,j} &\hspace{0.25cm} k+1 \leq i < s \\
& = c &\hspace{0.25cm} i = s 
\end{align*}

The maximal shifts in the resolution are:
\begin{align*}
M_i  &=  max_j a_{ij} & \hspace {0.25cm} 1\leq i \leq k \\
&= c - min_ja_{s-i,j} &\hspace{0.25cm} k+1 \leq i < s \\
& = c &\hspace{0.25cm} i = s 
\end{align*}

and $M_s=m_s=c.$ 
\\
\item Let $\alpha_{ij}=a_{ij}(c-a_{ij})$ for $ i \leq k$ with $p_i =min_j \alpha_{ij}=m_iM_{s-i}$ and  \\
$P_i =max_j \alpha_{ij}=M_im_{s-i}.$\\
 \end{enumerate}
 \end{remark}

\begin{definition} Given $(\alpha_1, \alpha_2, \ldots , \alpha_k)$ a sequence of real numbers, we denote the following Vandermonde determinants by
$$V_t=V_t(\alpha_1, \alpha_2, \ldots  , \alpha_{k}) = \left|
\begin{array}{cccc}
  1 & 1& .... & 1 \\
  \alpha_1&\alpha_2& ...&\alpha_k \\
  \alpha_1^2& \alpha_2^2 & ... &\alpha_k^2 \\
  \vdots & \vdots & ...&\vdots  \\
  \alpha_1^{k-2} & \alpha_2^{k-2} & ... & \alpha_k^{k-2} \\
  \alpha_1^{k-1+t}&\alpha_2^{k-1+t}& ... &\alpha_k^{k-1+t}\\
\end{array}\right|$$

 $$= \displaystyle \prod_{1 \leq j<i \leq k}(\alpha_i-\alpha_j)\displaystyle \sum_{\beta_1+\beta_2 + \ldots   \beta_k=t}(\alpha_1^{\beta_1}.\alpha_2^{\beta_2} \ldots  \alpha_k^{\beta_k})$$
 
 \end{definition}
 
  \begin{remark} $V_t(\alpha_1, \alpha_2, \ldots  , \alpha_{k}) \geq 0 $ if the sequence is in ascending order.
\end{remark} 

 As a convention, for any non-negative integers $n,p$, we 
 set the binomial coefficient 
 ${n \choose p} =0$ if $n <p$. 
  
The following binomial identities are essential to our theorems. In \cite {S},  Srinivasan showed

 \begin{lemma}  \label{lemma0} For all $k \geq 0 , c, a \geq 1$
 
  $(c-a)^{n} - a^{n} = \displaystyle \sum_{t=0}^{[\frac{n}{2}]}(-1)^t { n-t -1\choose t}  a^t(c-a)^t(c-2a)c^{n-2t-1} $
  
  $ (c-a)^{n} + a^{n} = \displaystyle \sum_{t=0}^{[\frac{n}{2}]} (-1)^t { n-t \choose t } a^t(c-a)^tc^{n-2t}+\displaystyle \sum_{t=1}^{[\frac{n}{2}]}  (-1)^t  {n-t-1 \choose  t -1 } a^t(c-a)^tc^{n-2t}$
  \end{lemma}

The proof goes along the same lines as in [\cite {S}, lemmas $2$-$3$].



\section{Hilbert Coefficients of Gorenstein Algebras.}

Let $R=K [x_1, \ldots x_n]$ and $I$ a graded ideal. Let $\mathbb F$ be the minimal resolution of $S=R/I$,
$$0\rightarrow \displaystyle\bigoplus_{j=1}^{b_s} R(-d_{sj}) \stackrel{\delta_s}{\rightarrow} \ldots\rightarrow \displaystyle\bigoplus_{j=1}^{b_i}R(-d_{ij})  \stackrel{\delta_i}{\rightarrow}  \ldots\rightarrow \displaystyle\bigoplus_{j=1}^{b_1}R(-d_{1j})\stackrel{\delta_1} \rightarrow R \rightarrow R/I \rightarrow 0 $$
 
\begin{theorem} (Peskine-Szpiro) Suppose $S$ is C-M then these shifts $d_{ij}$ are known to satisfy $[\mbox{\cite{PS}},1]$
$$
\displaystyle\sum_{i=1}^{s}(-1)^i \displaystyle\sum_{j=1}^{b_i}d_{ij}^k  = 
\left\{  \begin{array}{ccccc}  -1&& k=0\\
                             0&& 1 \leq k <s\\
                         (-1)^s s! e && k=s\\
                         \end{array}  \right.\ $$
\end{theorem}

These equations can be thought of as defining the multiplicity, $e= e_0(S)$.  In fact, the higher Hilbert Coefficients can also be expressed in terms of the shifts in the resolution\cite{E}. We include a simple proof  for the sake of completeness. 

\begin{theorem} \label{theorem1} $$ \begin{array}{ccccccc}
(-1)^s (s+l)!e_l &=& \displaystyle \sum_{r=0}^{l}(-1)^{l-r}\nu_{l-r} \displaystyle\sum _{i=0}^{s}(-1)^i  \displaystyle\sum_{j=1}^{b_i} d_{ij}^{s+r} \\
\end{array}$$
with $\nu_{l-r}=\displaystyle \sum_{1 \leq \xi_1 <\xi_2 <\ldots  <\xi_{l-r} \leq s+l-1}\xi_1.\xi_2. \ldots  \xi_{l-r}$ and $\nu_0=1.$
\end{theorem}

\begin{proof} We know that  the Hilbert function of $R/I$ is $$ \frac {\sum_{i=0}^{s} (-1)^i  \sum_{j=1}^{b_i} t^{d_{ij}}} {(1-t)^n}=\frac{Q(t)}{(1-t)^d}$$
where $d = dim$ $R/I = n-s$ and $\frac{Q^{(i)}(1)}{i!}=e_i$.

\noindent We get
\begin{equation} \label{equ0}
\displaystyle\sum_{i=0}^{s}(-1)^i \displaystyle \sum _{j=1}^{b_i}t^{d_{ij}}=Q(t)(1-t)^s \end{equation}
 \noindent We denote these two quantities by $S_{R/I}(t)$. We differentiate both sides $l+s$ times, and evaluate them at $t=1$. 
We first start by the right hand side
$$S^{(s+l)}_{R/I}(t)=  (-1)^s {s+l \choose
l} s! Q^{(l)}(t) + (1-t)P(t)$$
where $P(t)$ is a polynomial in $t$.  Evaluating at $t=1$:
\begin{align*} S^{(s+l)}_{R/I}(1)=&  (-1)^s {s+l \choose
l} s! Q^{(l)}(1) + 0\\
=&  (-1)^s (s+l)!e_l \end{align*}

\noindent On the other hand, $S_{R/I}(t) = \displaystyle\sum_{i=0}^{s}(-1)^i \displaystyle \sum _{j=1}^{b_i}t^{d_{ij}}$.  So

$$ \begin{array}{cccccc}
 S_{R/I}^{(l)}(1) &=& \displaystyle\sum _{i=0}^{s}(-1)^i  \displaystyle\sum_{j=1}^{b_i} {d_{ij}\choose
l} l!  \\
&=&  \displaystyle\sum _{i=0}^{s}(-1)^i  \displaystyle\sum_{j=1}^{b_i} \displaystyle \prod_{r=0}^{l-1}(d_{ij}-r)\\
&=& \displaystyle\sum _{i=0}^{s}(-1)^i  \displaystyle\sum_{j=1}^{b_i} \displaystyle \sum_{r=1}^{l}(-1)^{l-r}\nu_{l-r}d_{ij}^r
 \end{array}$$
 
 \noindent with $\nu_{l-r} = \displaystyle \sum_{1 \leq \xi_1 < \xi_2 < \dots < \xi_{l-r}<l-1} \xi_1\xi_2 \ldots  \xi_{l-r}$
\noindent and $\nu_0 =1$.
$$\begin{array}{ccccc}
S_{R/I}^{(s+l)}(1)& =& \displaystyle\sum _{i=0}^{s}(-1)^i  \displaystyle\sum_{j=1}^{b_i} \displaystyle \sum_{r=1}^{s+l}(-1)^{s+l-r}\nu_{s+l-r}d_{ij}^r \\
&= &  \displaystyle \sum_{r=1}^{s+l}(-1)^{s+l-r}\nu_{s+l-r} \displaystyle\sum _{i=0}^{s}(-1)^i  \displaystyle\sum_{j=1}^{b_i} d_{ij}^r \\
\end{array}$$

\noindent with $\nu_{s+l-r}=\displaystyle \sum_{1 \leq \xi_1 <\xi_2 <\ldots  <\xi_{s+l-r} \leq s+l-1}\xi_1.\xi_2. \ldots  \xi_{s+l-r}$ and $ \nu_0 = 1$.

\noindent Note that $\displaystyle\sum _{i=0}^{s}(-1)^i  \displaystyle\sum_{j=1}^{b_i} d_{ij}^r = 0$ when $r< s$ and

$$\begin{array}{ccccc}S_{R/I}^{(s+l)}(1) &= &  \displaystyle \sum_{r=s}^{s+l}(-1)^{s+l-r}\nu_{s+l-r} \displaystyle\sum _{i=0}^{s}(-1)^i  \displaystyle\sum_{j=1}^{b_i} d_{ij}^r \\
 &= &  \displaystyle \sum_{r=0}^{l}(-1)^{l-r}\nu_{l-r} \displaystyle\sum _{i=0}^{s}(-1)^i  \displaystyle\sum_{j=1}^{b_i} d_{ij}^{s+r}
\end{array}$$

\noindent and hence the result. 
\end{proof}

Let $I$ be Gorenstein. The minimal free resolution of $I$ is written as in (\ref{eq1}) and (\ref{eq2}) depending on whether the projective dimension of $I$ is even or odd. In \cite{S}, Srinivasan gave a more simplified expression for the multiplicity in both cases. She proved

\begin{theorem}(Srinivasan) \label{thmS1} Let $I$ be Gorenstein of grade $s=2k+1$ and the minimal graded resolution of $S=R/I$ be as in $(\ref{eq1})$. Then,
\begin{align*}   \displaystyle\sum_{i=1}^{k} \displaystyle\sum_{j=1}^{b_i} (-1)^i a_{ij}^t(c-a_{ij})^t(c-2a_{ij})& =0& \mbox{if} \hspace{0.25cm}1 \leq t <k\\
& = (-1)^k(2k+1)!e(S)&  \mbox{if} \hspace{0.25cm}t=k \\
&=-c &\mbox{if} \hspace{0.25cm}t=0
\end{align*} 
\end{theorem}

\begin{theorem}(Srinivasan) \label{thmS2} Let $I$ be Gorenstein of grade $s=2k$ and the minimal graded resolution of $S=R/I$ be as in $(\ref{eq2})$. Then,
\begin{align*}   \displaystyle\sum_{i=1}^{k} \displaystyle\sum_{j=1}^{b_i} (-1)^i a_{ij}^t(c-a_{ij})^t & =0& \mbox{if} \hspace{0.25cm}1 \leq t <k\\
& = (-1)^k\frac{(2k)!}{2}e(S)&  \mbox{if} \hspace{0.25cm}t=k\\
&=-1 &\mbox{if} \hspace{0.25cm} t=0
\end{align*} 
\end{theorem}

We extend these results  to all coefficients and we show

\begin{theorem} \label{th1} Let $I$ be Gorenstein of grade $s=2k+1$ and the minimal resolution of $S=R/I$ be as in $(\ref {eq1})$. Then
 $(-1)^k (s+l)!e_l $ is equal to \\
 \\
$ \displaystyle \sum_{ 0 \leq r \leq l} (-1)^{l-r}\nu_{l-r} \displaystyle \sum _{t=0}^{[\frac{r}{2}]}(-1)^{t}  {k+r-t \choose k+t} c^{r-2t} \displaystyle \sum_{i=1}^k \displaystyle \sum_{j=1}^{b_i} (-1)^{i} a_{ij}^{k+t}(c-a_{ij})^{k+t}(c-2a_{ij})$
\end{theorem}

\begin{proof}  Following the result of Theorem \ref{theorem1}, it suffices to show that
 $ \displaystyle\sum _{i=0}^{s}(-1)^i  \displaystyle\sum_{j=1}^{b_i} d_{ij}^{s+r}$ \\
 equals $-\displaystyle \sum_{i=1}^k(-1)^{k+i}  \displaystyle \sum_{j=1}^{b_i} \displaystyle \sum _{t = 0}^{[\frac{r}{2}]}(-1)^t {k+r-t \choose k+t} c^{r-2t}a_{ij}^{k+t}(c-a_{ij})^{k+t}(c-2a_{ij}).$ We have that $ \displaystyle\sum _{i=0}^{s}(-1)^i  \displaystyle\sum_{j=1}^{b_i} d_{ij}^{s+r} $
\begin{align*}&=-c^{s+r}+ \displaystyle \sum_{i=1}^k \displaystyle \sum_{j=1}^{b_j} a_{ij}^{s+r}(-1)^i+\displaystyle \sum_{i=1}^k \displaystyle \sum_{j=1}^{b_{j}} (-1)^{2k+1-i}(c-a_{ij})^{s+r}  \\
&= -c^{s+r} -\displaystyle \sum_{i,j}(-1)^i [(c-a_{ij})^{s+r}-a_{ij}^{s+r}] \hspace{2cm} \\
&=-c^{s+r} - \displaystyle \sum_{i,j}(-1)^i \displaystyle \sum _{t=0}^{[\frac{s+r}{2}]}(-1)^t {s+r-1-t \choose t } a_{ij}^t(c-a_{ij})^t(c-2a_{ij})c^{s+r-1-2t} \\
&\hspace{12cm} \mbox {by lemma}\hspace{0.1cm} \ref{lemma0}. \\
&=-c^{s+r} - \displaystyle \sum_{i,j}(-1)^i \displaystyle \sum _{t=0}^{k+[\frac{r+1}{2}]} (-1)^t  {2k+r-t \choose t} a_{ij}^t(c-a_{ij})^t(c-2a_{ij})c^{2k+r-2t}  \end{align*}
By theorem \ref{thmS1} the only remaining terms in the sum are $t =0$ and $t \geq k$, so  
\begin{align*}
\displaystyle\sum _{i=0}^{s}(-1)^i  \displaystyle\sum_{j=1}^{b_i} d_{ij}^{s+r}&=-\displaystyle \sum_{i,j}(-1)^i \displaystyle \sum _{t =k}^{k+[\frac{r+1}{2}]}(-1)^{t} {2k+r-t \choose t } a_{ij}^t(c-a_{ij})^t(c-2a_{ij})c^{2k+r-2t} \\
&=-\displaystyle \sum_{i,j}(-1)^i \displaystyle \sum _{t = 0}^{[\frac{r}{2}]}(-1)^{t+k} {k+r-t \choose k+t} a_{ij}^{k+t}(c-a_{ij})^{k+t}(c-2a_{ij})c^{r-2t} \end{align*}
\end{proof}

\begin{example}$ (-1)^k (s+1)!e_1= [-\nu_1+ {k+1\choose k} c] \displaystyle \sum_{i=1}^k \displaystyle \sum_{j=1}^{b_i} (-1)^{i} a_{ij}^{k}(c-a_{ij})^{k}(c-2a_{ij}).$\\
$ (-1)^k (s+2)!e_2= \left[\nu_2-\nu_1 {k+1 \choose k} c + {k+2 \choose k} c^2\right] \displaystyle \sum_{i=1}^k \displaystyle \sum_{j=1}^{b_i} (-1)^{i} a_{ij}^{k}(c-a_{ij})^{k}(c-2a_{ij})$\\

\hspace{6cm} $-  \displaystyle \sum_{i=1}^k \displaystyle \sum_{j=1}^{b_i} (-1)^{i} a_{ij}^{k+1}(c-a_{ij})^{k+1}(c-2a_{ij}).$
\end{example}

We now consider the case when $s$ is even.

\begin{theorem} \label{th2} Let $I$ be Gorenstein of grade $s=2k$ and the minimal resolution of $S=R/I$ be as in $(\ref {eq1})$. Then $(-1)^k (s+l)!e_l $ is equal to\\ 
\\ 
$ \displaystyle \sum_{\tiny {\begin{array}{cc}t=0 \\ 0 \leq r \leq l \end{array}}} ^{[\frac{r}{2}]} \normalsize (-1)^{l-r+t} \nu_{l-r}\left[ {k+r-t \choose k+t} +{k+r-t -1\choose k+t-1} \right]c^{r-2t} \displaystyle \sum_{i=1}^k \displaystyle \sum_{j=1}^{b_i} (-1)^{i} a_{ij}^{k+t}(c-a_{ij})^{k+t} $
In the summation $j$ runs from $1$ to $b_i$ if $i < k$ and from $1$ to $b_k/2$ if $i=k$.
\end{theorem}
\begin{proof} We proceed the same way as the proof of theorem \ref{th1}. 
We have 
\begin{align*} \displaystyle\sum _{i=0}^{s}(-1)^i  \displaystyle\sum_{j=1}^{b_i} d_{ij}^{s+r}&= c^{s+r}+ \displaystyle \sum_{i=1}^k \displaystyle \sum_{j=1}^{b_j} a_{ij}^{s+r}(-1)^i+ \displaystyle \sum (-1)^{2k-i}(c-a_{ij})^{s+r}\\
& =c^{s+r} +\displaystyle \sum_{i,j}(-1)^i \left[(c-a_{ij})^{s+r}+a_{ij}^{s+r}\right]. \\
&=c^{2k+r} + \displaystyle \sum_{i,j}(-1)^i \left[\displaystyle \sum_{t=0}^{k+[\frac{r}{2}]} (-1)^t {2k+r-t \choose t} a_{ij}^t(c-a_{ij})^tc^{2k+r-2t} \right. \\
& \hspace{3cm} \left. +\displaystyle \sum_{t=1}^{k+[\frac{r}{2}]} (-1)^t { 2k+r-1-t \choose t-1} a_{ij}^t(c-a_{ij})^tc^{2k+r-2t}\right] \\
&\hspace{9.5cm} \mbox{by lemma \ref{lemma0}.}  \end{align*}

By theorem \ref{thmS2} the only remaining terms in the sum are $t =0$ and $t \geq k$, we obtain that $ \displaystyle\sum _{i=0}^{s}(-1)^i  \displaystyle\sum_{j=1}^{b_i} d_{ij}^{s+r}=$\\  
 $\displaystyle \sum_{i,j}(-1)^i \left[ \displaystyle \sum_{t =0}^{[\frac{r}{2}]} (-1)^{t+k} {k+r-t \choose t+k} a_{ij}^{t+k}(c-a_{ij})^{t+k}c^{r-2t} \right. \\$

\hspace{5cm}  $\left. +\displaystyle \sum_{t = 0}^{[\frac{r}{2}]}(-1)^{t+k} {k+r-1-t \choose t+k-1} a_{ij}^{t+k}(c-a_{ij})^{t+k}c^{r-2t} \right]$. \\ 

 \end{proof}
 
 \begin{example}
 
 $ (-1)^k (s+1)!e_1= \left[ -\nu_1\left( {k\choose k} + {k-1\choose k-1} \right)+ \left(  {k+1\choose k}+{k\choose k-1} \right)c\right].$\\
 
\hspace{9cm} $\displaystyle \sum_{i=1}^k \displaystyle \sum_{j=1}^{b_i} (-1)^{i} a_{ij}^{k}(c-a_{ij})^{k}$\\
\\
\\
$ (-1)^k (s+2)!e_2=\left[\nu_2\left( {k\choose k}+{k-1\choose k-1} \right)-\nu_1 \left( {k+1\choose k}+{ k\choose k-1}\right)c  \right. + \left. \left( {k+2\choose k}+{ k+1\choose k-1} \right)c^2\right].$

\hspace{1cm}$\displaystyle \sum_{i=1}^k \displaystyle \sum_{j=1}^{b_i} (-1)^{i} a_{ij}^{k}(c-a_{ij})^{k} -\left[ {k+1\choose k+1}+ {k\choose k} \right] \displaystyle \sum_{i=1}^k \displaystyle \sum_{j=1}^{b_i} (-1)^{i} a_{ij}^{k+1}(c-a_{ij})^{k+1}$
\end{example}

\section{Bounds for the coefficients with quasi-pure resolutions.}

 \begin{definition} For any ordered $s$-tuple of positive integers, 
$$f_l(y_{1 }, \ldots , y_{s }) = 
 \displaystyle \sum_{1 \leq i_1 \leq \ldots  i_l \leq s}(\displaystyle \prod_{t=1}^l y_ {i_t }-(i_{t}+t-1)) , 1\leq l \leq s$$
 and $f_0 = 1$. 
 \end{definition}  
In this section, we prove 

\begin{theorem} \label{main}  If $S$ is a homogeneous Gorenstein Algebra with quasi-pure resolution of length $s = 2k$ or $2k+1$.  Then
\vskip 1cm
\noindent \small  $f_{l}(m_1 \ldots  m_kM_{k+1} \ldots  M_s) \frac{m_1 \ldots  m_kM_{k+1} \ldots  M_s }{(s+l)!}\leq e_l(S) 
\leq f_{l}(M_1 \ldots  M_km_{k+1} \ldots  m_s)\frac{M_1 \ldots  M_km_{k+1} \ldots  m_s}{(s+l)!}$
\end{theorem}

\begin{remark} \begin{enumerate} \item $M_n=m_n=c$.
\item These bounds are strictly stronger than the bounds in the conjecture found by Herzog and Zheng in \cite{HZ}. \end{enumerate} 
\end{remark}
Let   $S=R/I$ . Then $S$ has a $R-$ free resolution of length $s= 2k+1$ or $s=2k$ .  The first half of the resolution is given below.
\begin{equation} \label{eq3} \displaystyle\sum_{j=1}^{b_k}R(-a_{kj})  \rightarrow  \ldots \rightarrow \displaystyle\sum_{j=1}^{b_1}R(-a_{1j}) \rightarrow R 
, \hspace {1cm}
 s =2k+1  \end{equation}
and 
\begin{equation}\label{eq4} \displaystyle\sum_{j=1}^{b_k/2=r_k}R(-(c-a_{kj})) \oplus \displaystyle\sum_{j=1}^{b_k/2=r_k}R(-a_{kj})  \rightarrow  \ldots \rightarrow \displaystyle\sum_{j=1}^{b_1}R(-a_{1j}) \rightarrow R, \hspace {1cm} s=2k\end{equation}

Thus, we let $r_i = b_i, i\neq k$ and $r_k = b_k$ if $s$ is odd and $r_k = \frac {b_k}{2}$ if $s$ is even. 

Without loss of generality we may take $a_{i1} \leq a_{i2} \leq \ldots  a_{ib_i}$ for all $i$.
\\
If $s=2k$, we pick $a_{kj}$ so that $c-a_{kr_k} \geq a_{kr_k}$. The symmetry of the resolution and the exactness criterion forces $b_k$ to be even.
Srinivasan showed in [\cite{S}, 5], that
\begin{lemma}\label{lemS} If $S$ is Gorenstein with a quasi-pure resolution then $c \geq 2a_{ij}$ for all $i,j$.
\end{lemma}

\begin{proof} Since quasi-purity means the $a_{ij}$ increase  with $i$, we just need to check \mbox{$c \geq 2a_{kr_k}$}.
If $s=2k+1$ then $c-a_{kr_k}=m_{k+1} \geq M_k =a_{kr_k}$. So $c \geq 2a_{kr_k}$. If $s=2k$, then $c \geq 2a_{kr_k}$ by choice and hence the result.
\end{proof}

To be able to prove theorem \ref{main} we need to consider two different determinants depending  on whether $s$ is even or odd. \\
Suppose $s$ is odd. Let
$$M_{t}= \left| \begin{array}{cccccc}
\displaystyle \sum_{j=1}^{b_1}\alpha_{1j}(c-2a_{1j}) \ldots  &\displaystyle \sum_{j=1}^{b_i}\alpha_{ij}(c-2a_{ij}) \ldots   & \displaystyle\sum_{j=1}^{b_k}\alpha_{kj}(c-2a_{kj})\\
\displaystyle \sum_{j=1}^{b_1}\alpha_{1j}^2(c-2a_{1j}) \ldots  &\displaystyle \sum_{j=1}^{b_i}\alpha_{ij}^2(c-2a_{ij}) \ldots   & \displaystyle\sum_{j=1}^{b_k}\alpha_{kj}^2(c-2a_{kj})\\

\vdots &\vdots& \vdots& \\

\displaystyle \sum_{j=1}^{b_1}\alpha_{1j}^{k-1}(c-2a_{1j}) \ldots & \displaystyle \sum_{j=1}^{b_i}\alpha_{ij}^{k-1}(c-2a_{ij}) \ldots   & \displaystyle\sum_{j=1}^{b_k}\alpha_{kj}^{k-1}(c-2a_{kj})\\
\displaystyle \sum_{j=1}^{b_1}\alpha_{1j}^{k+t}(c-2a_{1j}) \ldots & \displaystyle \sum_{j=1}^{b_i}\alpha_{ij}^{k+t}(c-2a_{ij}) \ldots   & \displaystyle\sum_{j=1}^{b_k}\alpha_{kj}^{k+t}(c-2a_{kj})\\
\end{array} \right|$$
where $\alpha_{ij} = a_{ij}(c-a_{ij})$. Then, $M_{t} = \displaystyle \sum_{1 \leq j_i \leq b_i} \displaystyle \prod_{i=1}^k \alpha_{ij_i}(c-2a_{ij_i}).V_t(\alpha_{1j_1}, \ldots \alpha_{kj_k}).$\\
Now consider $$\displaystyle \sum_{r=0}^l(-1)^{l-r}\nu_{l-r} \displaystyle \sum _{t=0}^{[r/2]}(-1)^{t}  { k+r-t \choose k+t} c^{r-2t} M_t $$\\
It is equal to \\
\\
$ \displaystyle \sum_{1 \leq j_i \leq b_i} \displaystyle \prod_{i=1}^k \alpha_{ij_i}(c-2a_{ij_i}) V(\alpha_{1j_1} \ldots \alpha_{kj_k}) .$

\hspace{3cm} $\displaystyle \sum_{r=0}^l (-1)^{l-r}\nu_{l-r} \displaystyle \sum _{t=0}^{[r/2]}(-1)^{t} {k+r-t \choose k+t} c^{r-2t}  \displaystyle \sum_{ \sum \beta_i=t} \prod_{i=1}^k \alpha_{ij_i}^{\beta_i}$ \\
\\
We thank L\'{a}szl\'{o} Sz\'{e}kely for his help with the following lemma.

\begin{lemma} \label{lem4.1} We have for all $c,a_i >0$ and $\alpha_i=a_i(c-a_i)$ \\
$\displaystyle \sum _{t=0}^{[r/2]}(-1)^{t} {k+r-t \choose k+t} c^{r-2t}\displaystyle \sum_{ \sum \beta_i=t}\prod_{i=1}^k \alpha_{i}^{\beta_i}= \displaystyle \sum_{\beta_1+ \dots + \beta_{2k+1}=r }  \prod_{i=1}^k a_i^{\beta_i}(c-a_i)^{\beta_{k+i}}c^{\beta_{2k+1}}$
\end{lemma}

 \begin{proof} $\displaystyle \sum_{\beta_1+ \dots + \beta_{2k+1}=r }  \prod_{i=1}^k a_i^{\beta_i}(c-a_i)^{\beta_{k+i}}c^{\beta_{2k+1}}$ is the coefficient of $x^r$ in
 
 \begin{align*} {\frac {1}{1-cx}}\prod_{i=1}^k \frac{1}{1-a_ix} \frac{1}{1-(c-a_i)x}& =\frac{1}{1-cx} \prod_{i=1}^k \frac{1}{1-(cx-a_i(c-a_i)x^2)} \\
& =\displaystyle \sum_{\gamma_1,\ldots , \gamma_{k+1}} \prod_{i=1}^k \left( cx-a_i(c-a_i)x^2 \right) ^{\gamma_i} \left( cx \right) ^{\gamma_{k+1}} \end{align*}
 the coefficient of $x^r$ is in this last expression is:
 
$$ \displaystyle \sum _{\sum_{i=1}^k (2\beta_i+  \gamma_i -\beta_i)+  \gamma_{k+1} = r} \prod_{i=1}^k {\gamma_i \choose \beta_i} c^{\gamma_i -\beta_i}(-1)^{\beta_i}\left(a_i(c-a_i)\right)^{\beta_i} c^{\gamma_{k+1}} =$$

$$ \displaystyle \sum _{\sum_{i=1}^k ( \beta_i+  \gamma_i )+  \gamma_{k+1} = r} \prod_{i=1}^k {\gamma_i \choose \beta_i} c^{\gamma_i -\beta_i}(-1)^{\beta_i}\left(a_i(c-a_i)\right)^{\beta_i} c^{\gamma_{k+1}} =$$

$$ \displaystyle \sum_{t \geq 0}c^{r-2t}(-1)^t   \displaystyle \sum _{  \sum_{i=1}^{k+1}\gamma_i =r-t}  \displaystyle \sum_{\beta_1+ \ldots  \beta_k=t} \prod_{i=1}^k \left(a_i(c-a_i)\right)^{\beta_i}   {\gamma_i \choose \beta_i} =$$

$$\displaystyle \sum_{t \geq 0}  c^{r-2t}(-1)^t \displaystyle \sum_{\beta_1+ \ldots  \beta_k=t} \prod_{i=1}^k \left(a_i(c-a_i) \right)^{\beta_i} \displaystyle \sum _{  \sum_{i=1}^{k+1}\gamma_i =r-t}\displaystyle \prod_{i=1}^k {\gamma_i \choose \beta_i}.$$
It remains to show that,  $\displaystyle \sum _{  \sum_{i=1}^{k+1}\gamma_i =r-t}\displaystyle \prod_{i=1}^k { \gamma_i \choose \beta_i} = {r-t+k \choose t+k}$.   
This can be proved by induction on $n,t, k$. Alternatively,
 consider the negative binomial theorem $$(1-x)^{-(\beta_{i}+1)} = \displaystyle \sum_n {n + \beta_i \choose \beta_i} x^n.$$
$$x^{\beta_i}(1-x)^{-(\beta_{i}+1)} = \displaystyle \sum_n {n + \beta_i\choose \beta_i} x^{n+\beta_i}= \displaystyle \sum_{\gamma_i} { \gamma_i\choose \beta_i} x^{\gamma_i},$$
so that $\displaystyle \sum_{\gamma_1+\ldots  \gamma_{k+1}=r-t}\displaystyle \prod_{i=1}^k { \gamma_i \choose \beta_i}$ is the coefficient of $x^{r-t-\gamma_{k+1}}$ in \\
$$ \displaystyle \prod_{i=1}^k \frac{x^{\beta_i}}{(1-x)^{\beta_i+1}}=\frac{x^{t}}{(1-x)^{t+k}}.$$\\ This in turn, is equal to the coefficient of $x^{r-2t-\gamma_{k+1}}$ in  $\frac{1}{(1-x)^{t+k}}$.  So, 
\begin{align*}  \displaystyle \sum_{\gamma_1+\ldots  \gamma_{k+1}=r-t}\displaystyle \prod_{i=1}^k { \gamma_i \choose \beta_i}=&\displaystyle \sum _{\gamma_{k+1}=0}^{r-t} {(r-\gamma_{k+1}-2t)+(t+k-1)\choose t+k-1}\\=& \displaystyle \sum _{\gamma_{k+1}=0}^{r-t} { (r-t+k-1-\gamma_{k+1}\choose t+k-1}\\
= &{r-t+k\choose t+k}.\\
 \end{align*} \end{proof}

 Now, when $s$ is even we consider the following determinant:

$$N_t= \left| \begin{array}{cccccc}
\displaystyle \sum_{j=1}^{r_1}\alpha_{1j} \ldots  &\displaystyle \sum_{j=1}^{r_i}\alpha_{ij}\ldots   & \displaystyle\sum_{j=1}^{r_k}\alpha_{kj}\\
\displaystyle \sum_{j=1}^{r_1}\alpha_{1j}^2 \ldots  &\displaystyle \sum_{j=1}^{r_i}\alpha_{ij}^2 \ldots   & \displaystyle\sum_{j=1}^{r_k}\alpha_{kj}^2\\

\vdots &\vdots& \vdots& \\

\displaystyle \sum_{j=1}^{r_1}\alpha_{1j}^{k-1} \ldots & \displaystyle \sum_{j=1}^{r_i}\alpha_{ij}^{k-1} \ldots   & \displaystyle\sum_{j=1}^{r_k}\alpha_{kj}^{k-1}\\
\displaystyle \sum_{j=1}^{r_1}\alpha_{1j}^{k+t} \ldots & \displaystyle \sum_{j=1}^{r_i}\alpha_{ij}^{k+t} \ldots   & \displaystyle\sum_{j=1}^{r_k}\alpha_{kj}^{k+t}\\
\end{array} \right|$$
with $\alpha_{ij} = a_{ij}(c-a_{ij})$, $r_i= b_i$ for $i=1, \ldots  k-1$ and $r_k = b_k/2$. Then, \\
$N_t = \displaystyle \sum_{1 \leq j_i \leq b_i}  \displaystyle \prod_{i=1}^k \alpha_{ij_i}V_t(\alpha_{1j_1}, \ldots , \alpha_{kj_k}).$\\
\\
Now consider
$$\displaystyle \sum_{0 \leq r \leq l} (-1)^{l-r}\nu_{l-r} \displaystyle \sum _{t = 0}^{[r/2]}(-1)^{t} \left[{k+r-t \choose k+t} +{k+r-t -1 \choose k+t-1} \right]c^{r-2t}N_t$$
It is equal to \\
 $\displaystyle \sum_{ j_i } \displaystyle \prod_{i=1}^k \alpha_{ij_i}V(\alpha_{1j_1} \ldots  \alpha_{kj_k}).$
 
\noindent \hspace{0.15cm} $\displaystyle \sum_{0 \leq r \leq l} (-1)^{l-r}\nu_{l-r} \displaystyle \sum _{t = 0}^{[r/2]}(-1)^{t} \left[ {k+r-t \choose k+t} + {k+r-t -1\choose k+t-1} \right]c^{r-2t} \displaystyle \sum_{ \sum \beta_i=t} \prod_{i=1}^k \alpha_{ij_i}^{\beta_i}$\\
\\
The following is the version of lemma \ref{lem4.1} for the case where $s$ is even.
\begin{lemma} \label{lem4.2} For all $c, a_i >0$ and $\alpha_i = a_i(c-a_i)$, we have \\
$\displaystyle \sum _{t = 0}^{[r/2]}(-1)^{t} \left[{k+r-t \choose k+t} +{ k+r-t -1 \choose k+t-1} \right]c^{r-2t}\displaystyle \sum_{ \sum \beta_i=t} \prod_{i=1}^k \alpha_{i}^{\beta_i} =$

\hspace{5.5cm} $\displaystyle \sum_{\beta_1+ \ldots \beta_s =r}  \prod_{i=1}^{k-1} a_i^{\beta_i}  (c-a_i)^{\beta_{k+i}} c^{\beta_{2k}} (a_k^{\beta_k}+(c-a_k)^{\beta_k})$.
\end{lemma}

 \begin{proof} The proof goes along the same lines as the odd case. \\
 $\displaystyle \sum_{\beta_1+ \ldots \beta_s =r}  \prod_{i=1}^{k-1} a_i^{\beta_i}  (c-a_i)^{\beta_{k+i}} c^{\beta_{2k}} \left(a_k^{\beta_k}+(c-a_k)^{\beta_k}\right) $
 is the coefficient of $x^r$ in
$${\frac {1}{1-cx}}\displaystyle \prod_{i=1}^{k-1} \frac{1}{1-a_ix} \frac{1}{1-(c-a_i)x} \left(\frac{1}{1-a_kx}+ \frac{1}{1-(c-a_k)x}\right)=$$
$${\frac {1}{1-cx}} \displaystyle \prod_{i=1}^{k} \frac{1}{1-a_ix} \frac{1}{1-(c-a_i)x} \left[1-(c-a_k)x+1-a_kx \right] =$$
$$ {\frac {1}{1-cx}}  \displaystyle \prod_{i=1}^{k} \frac{1}{1-a_ix} \frac{1}{1-(c-a_i)x} \left[ 2-cx \right]= $$
$$    \displaystyle \prod_{i=1}^{k} \frac{1}{1-a_ix} \frac{1}{1-(c-a_i)x} \left[1+{\frac{1}{1-cx}}\right]. $$

By an argument similar to the proof of lemma \ref {lem4.1} applied to the two sums separately, we see that this is the coefficient of $x^r$ in
$$  \displaystyle \sum _{t=0}^{[r/2]}(-1)^{t} {k-1+r-t \choose k-1+t} c^{r-2t} \displaystyle \sum_{ \sum \beta_i=t} \prod_{i=1}^k \alpha_{i}^{\beta_i} +  \displaystyle \sum _{t=0}^{[r/2]}(-1)^{t} {k +r-t \choose k +t} c^{r-2t} \displaystyle \sum_{ \sum \beta_i=t} \prod_{i=1}^k \alpha_{i}^{\beta_i}$$
 $=\displaystyle \sum _{t = 0}^{[r/2]}(-1)^{t} \left[{k+r-t \choose k+t} +{ k+r-t -1 \choose k+t-1} \right]c^{r-2t}\displaystyle \sum_{ \sum \beta_i=t} \prod_{i=1}^k \alpha_{i}^{\beta_i} .$
 
 \end{proof}
  
Now in both cases whether $s$ is even or odd we have
  \begin{lemma} \label{lem4.6}  
$\displaystyle \sum_{0 \leq r \leq l} (-1)^{l-r}\nu_{l-r} \displaystyle \sum_{\beta_1+ \dots + \beta_{2k+1}=r} \prod_{i=1}^k a_{ij_i}^{\beta_i}(c-a_{ij_i})^{\beta_{k+i}}c^{\beta_{2k+1}}$  
and \\
 $\displaystyle \sum_{0 \leq r \leq l} (-1)^{l-r}\nu_{l-r} \displaystyle \sum_{\beta_1+ \ldots \beta_s =r}  \prod_{i=1}^{k-1} a_{ij_i}^{\beta_i}  (c-a_{ij_i})^{\beta_{k+i}} c^{\beta_{2k}} (a_{kj_k}^{\beta_k}+(c-a_{kj_k})^{\beta_k})$ are both equal to\\
 $\displaystyle \sum_{1 \leq i_1 \leq \ldots  i_l \leq s}(\displaystyle \prod_{t=1}^l d_{i_tj_{i_t}}-(i_{t}+t-1)) .$
  \end{lemma}
 
\begin{proof}
For any given $r-$tuples $1 \leq \alpha_1 \leq \ldots \leq \alpha_r \leq s$ with $0 \leq r \leq l$ we will have $1 \leq \beta_1 \leq \ldots \leq \beta_{l-r} \leq s$ such that $\left\{ \alpha_1, \ldots \alpha_r \right\} \cup \left\{ \beta_1, \ldots \beta_{l-r} \right\}$ is equal to $\left\{i_1, \ldots i_l \right\}$.  In the product $\displaystyle \sum_{1 \leq i_1 \leq \ldots  i_l \leq s}(\displaystyle \prod_{t=1}^l d_{i_tj_{i_t}}-(i_{t}+t-1))$, the coefficient of $\displaystyle \prod_{1\leq  \alpha_1 \leq \ldots \leq \alpha_r \leq s}d_{\alpha_tj_{\alpha_t}}$  is $\displaystyle \sum_r (-1)^{l-r} \displaystyle \prod_{1 \leq \beta_1 < \ldots <\beta_{l-r} \leq s+l-1}\beta_1\ldots \beta_{l-r}  = \displaystyle \sum_r (-1)^{l-r}\nu_{l-r}$, since $i_t+t-1$ is strictly increasing until $i_l+l-1$ .
 Further,  if $s$ is odd $d_{ij} = a_{ij} , i\leq k$ and $d_{ij} = c- a_{(2k+1-i),j}, i> k$  and if $s$ is even  $d_{ij} = a_{ij} , i< k$ ;  $d_{ij} = c- a_{(2k-i),j}, i> k$ and $d_{kj} = a_{kj}$ or $ c-a_{kj}$.  Hence, 
$$\begin{matrix}
  \displaystyle \sum_{1 \leq i_1 \leq \ldots  i_l \leq s}(\displaystyle \prod_{t=1}^l d_{i_tj_{i_t}}-(i_{t}+t-1))  =&\\
 =  \displaystyle \sum_{0 \leq r \leq l} (-1)^{l-r}\nu_{l-r} \displaystyle \sum_{\beta_1+ \dots + \beta_{2k+1}=r} \prod_{i=1}^k a_{ij_i}^{\beta_i}(c-a_{ij_i})^{\beta_{k+i}}c^{\beta_{2k+1}} & , s= 2k+1 \\
= \displaystyle \sum_{0 \leq r \leq l} (-1)^{l-r}\nu_{l-r} \displaystyle \sum_{\beta_1+ \ldots \beta_s =r}  \prod_{i=1}^{k-1} a_{ij_i}^{\beta_i}  (c-a_{ij_i})^{\beta_{k+i}} c^{\beta_{2k}} (a_{kj_k}^{\beta_k}+(c-a_{kj_k})^{\beta_k})&, s= 2k\\
 \end{matrix}.
 $$  This completes the proof. 
\end{proof}
\begin{remark}\label{rem4.6} $\displaystyle \sum_{1 \leq i_1 \leq \ldots  i_l \leq s}(\displaystyle \prod_{t=1}^l d_{i_tj_{i_t}}-(i_{t}+t-1)) =    f_l(a_{1j_1}, \ldots , c)$, when $s=2k+1$ is odd.  However, when $s=2k$,   $d_{kj} $ can equal either $a_{kj}$ or $(c-a_{kj})$ and hence    

 $\displaystyle \sum_{1 \leq i_1 \leq \ldots  i_l \leq s}(\displaystyle \prod_{t=1}^l d_{i_tj_{i_t}}-(i_{t}+t-1)) = f_l(a_{1j_1}, \ldots a_{kj_k}, {\ldots c)+f_l(a_{1j_1}, \ldots , c- a_{kj_k}, \ldots c)}.$
 
 Thus, when $s=2k+1$, 
 
  $\displaystyle \sum_{0 \leq r \leq l} (-1)^{l-r}\nu_{l-r} \displaystyle \sum_{\beta_1+ \dots + \beta_{2k+1}=r} \prod_{i=1}^k a_{ij_i}^{\beta_i}(c-a_{ij_i})^{\beta_{k+i}}c^{\beta_{2k+1}}=  f_l(a_{1j_1}, \ldots , c)$ 
  
  and when $s=2k$, 
  
  $\displaystyle \sum_{r=0}^l  \displaystyle \sum_{\beta_1+ \ldots \beta_s =r} (-1)^{l-r}\nu_{l-r}  \prod_{i=1}^{k-1} a_{ij_i}^{\beta_i}  (c-a_{ij_i})^{\beta_{k+i}} c^{\beta_{2k}} (a_{kj_k}^{\beta_k}+(c-a_{kj_k})^{\beta_k})= $
  
  \hspace {4cm} $f_l(a_{1j_1}, \ldots a_{kj_k}, {\ldots c)+f_l(a_{1j_1}, \ldots , c- a_{kj_k}, \ldots c)}.$
\end{remark}
  
\begin {remark} \label {f_i} Suppose $R/I$ has a quasi-pure resolution. Since the Hilbert function of $R/I$ is unaltered by any cancellations, we may assume $d_{ij} > d_{{i-1j}}$ for all $i$.
Since $d_{ij} \geq i$ for all $i$, we get $d_{ij}-i \geq d_{i-1}j - (i-1)$ and hence $d_{p_j}-p \geq d_{q_j}-q$ for all $p \geq q$.
Now assume that not all factors in $\displaystyle \prod_{t=1}^l(d_{i_tj_{i_t}}-(i_t+t-1))$ are positive and let $p$ be the smallest integer with $d_{i_pj_{i_p}}- (i_p+p-1) < 0$. Then $p>1$ and $d_{i_{p-1}j_{i_{p-1}}}- (i_{p-1}+p-2) >0$. It follows that $$d_{i_{p-1}j_{i_{p-1}}}- (i_{p-1}+p-2) - d_{i_pj_{i_p}}- (i_p+p-1) \geq 2 $$
or equivalently $$i_p-i_{p-1} \geq d_{i_pj_{i_p}} - d_{i_{p-1}}j_{i_{p-1}}+1$$ 
$$d_{i_{p-1}}j_{i_{p-1}}-i_{p-1} \geq d_{i_pj_{i_p}} - i_p +1 > d_{i_pj_{i_p}}- i_p$$
which is a contradiction, and we get that $\displaystyle \prod_{t=1}^l(d_{i_tj_{i_t}}-(i_t+t-1)) \geq 0$.\\
    Thus,  $ f_l(d_{1j_1}, \ldots , d_{sj_s}) \geq 0$, for all $s$-tuples $(d_{1j_1}, \ldots , d_{sj_s})$, provided 
  the resolution is quasi-pure. 
  \end{remark}

\begin{proof}\textit{of theorem \ref{main}}.  We have $\alpha_{ij} = a_{ij}(c-a_{ij}) $ and $\alpha_{ij} -\alpha'_{ij} = (a_{ij}-a'_{ij'})(c-a_{ij}-a'_{ij'})$.  By the quasi-purity of the resolution of $S$ and lemma \ref{lemS}, $c\geq 2a_{ij}$ for all $i,j$. So $a_{i1} \leq a_{i2} \leq \ldots  \leq a_{ib_i}$ implies that $\alpha_{i1} \leq \alpha_{i2} \leq \ldots  \leq\alpha_{ib_i}  $.\\
We also have for all $1 \leq i \leq k$, $p_i = \mbox{min}_j \alpha_{ij} =m_iM_{s-i}$ and 
$P_i = \mbox{max}_j \alpha_{ij} =M_im_{s-i}.$\\
\\
We treat the even and odd cases separately. \\
Case $1$. $s=2k+1$ is odd. The resolution starts as in (\ref{eq3}) and by lemmas \ref {lem4.1}, \ref{lem4.6} and remark \ref {rem4.6} we have that \\
$ \displaystyle \sum_{1 \leq j_i \leq b_i} \displaystyle \prod_{i=1}^k \alpha_{ij_i}(c-2a_{ij_i}) V(\alpha_{1j_1}, \ldots \alpha_{kj_k}) .$

\hspace{3cm} $\displaystyle \sum_{r=0}^l (-1)^{l-r}\nu_{l-r} \displaystyle \sum _{t=0}^{[r/2]}(-1)^{t} { k+r-t \choose k+t} c^{r-2t}  \displaystyle \sum_{ \sum \beta_i=t} \prod_{i=1}^k \alpha_{ij_i}^{\beta_i} =$ \\
$$ \displaystyle \sum_{1 \leq j_i \leq b_i} \displaystyle \prod_{i=1}^k \alpha_{ij_i}(c-2a_{ij_i}).V(\alpha_{1j_1} \ldots \alpha_{kj_k}) f_l(a_{1j_1}, \ldots , c)$$
 By remark \ref {f_i}, $f_l(d_{1j_1}, \ldots  d_{sj_s}) \geq 0$. Further, $c \geq 2a_{ij_i}$by lemma \ref {lemS} and   \mbox{$V(\alpha_{1j_1}, \ldots , \alpha_{kj_k}) \geq 0$}, for $\alpha_{1j_1} \leq \ldots  \leq \alpha_{kj_k}$.   So \\
$\displaystyle \sum_{j_i}  \displaystyle \prod_{i=1}^k p_i(c-2a_{ij_i})V(\alpha_{1j_1}, \ldots , \alpha_{kj_k}) f_l\left( m_1, \ldots  m_k, M_{k+1}, \ldots  M_{s}\right)  \leq $ 

\hspace{2.5 cm}$\displaystyle \sum_{ j_i } \displaystyle \prod_{i=1}^k \alpha_{ij_i}(c-2a_{ij_i}).V(\alpha_{1j_1} \ldots \alpha_{kj_k}) f_l(a_{1j_1}, \ldots , c )$ 

\hspace{4cm} $ \leq \displaystyle \sum_{j_i}   \displaystyle \prod_{i=1}^k P_i(c-2a_{ij_i})V(\alpha_{1j_1}, \ldots , \alpha_{kj_k}) f_l\left( M_1, \ldots  M_k, m_{k+1}, \ldots  m_{s}\right)$

which is the same as

$f_l\left( m_1, \ldots  m_k, M_{k+1}, \ldots  M_{s}\right)\displaystyle \prod_{i=1}^k p_i det(Q) \leq$\\

\hspace{3 cm}$\displaystyle \sum_{ j_i } \displaystyle \prod_{i=1}^k \alpha_{ij_i}(c-2a_{ij_i}).V(\alpha_{1j_1} \ldots \alpha_{kj_k}) f_l(a_{1j_1}, \ldots , c)$ \\

 \hspace{6cm} $\leq f_l\left( M_1, \ldots  M_k, m_{k+1}, \ldots  m_{s}\right) \displaystyle \prod_{i=1}^k P_i det(Q)$\\
 \\
 where 
 $$Q =  \left( \begin{array}{cccccc}
\displaystyle \sum_{j=1}^{b_1}(c-2a_{1j}) \ldots  &\displaystyle \sum_{j=1}^{b_i}(c-2a_{ij}) \ldots   & \displaystyle\sum_{j=1}^{b_k}(c-2a_{kj})\\
\displaystyle \sum_{j=1}^{b_1}\alpha_{1j}(c-2a_{1j}) \ldots  &\displaystyle \sum_{j=1}^{b_i}\alpha_{ij}(c-2a_{ij}) \ldots   & \displaystyle\sum_{j=1}^{b_k}\alpha_{kj}(c-2a_{kj})\\

\vdots &\vdots& \vdots& \\

\displaystyle \sum_{j=1}^{b_1}\alpha_{1j}^{k-1}(c-2a_{1j}) \ldots & \displaystyle \sum_{j=1}^{b_i}\alpha_{ij}^{k-1}(c-2a_{ij}) \ldots   & \displaystyle\sum_{j=1}^{b_k}\alpha_{kj}^{k-1}(c-2a_{kj})\\
\end{array} \right)$$\\
$det Q >0$ since at least one of the $V(\alpha_{1j_1}, \ldots , \alpha_{kj_k}) > 0$ and  $\displaystyle \prod_{i=1}^k(c-2a_{ij_i})>0.$
Replacing the last column by the alternating sums of columns in $Q$ and using theorems \ref{thmS1} and \ref{th2}, we get
$$det Q = det \left(\begin{array}{cccc}\ldots  &-c \\
L& 0\end{array} \right) = c.detL$$
So 

$f_l\left( m_1, \ldots  m_k, M_{k+1}, \ldots  M_{s}\right) \displaystyle \prod_{i=1}^k p_i cdet(L) \leq $\\

\hspace{3 cm}$\displaystyle \sum_{ j_i } \displaystyle \prod_{i=1}^k \alpha_{ij_i}(c-2a_{ij_i}).V(\alpha_{1j_1} \ldots \alpha_{kj_k}) f_l(a_{1j_1}, \ldots , c)  $\\

\hspace{6cm} $\leq f_l\left( M_1, \ldots  M_k, m_{k+1}, \ldots  m_{s}\right)\ \displaystyle \prod_{i=1}^k P_i cdet(L)$\\
\\
On the other hand, we start with $M_{t}$ again. Replacing the last column of $M_{t}$ by alternating sums of the columns and using theorem \ref{th1} we get

$$M_{t} = \left|\begin{array}{cccc} L  &&&0 \\
\ldots  &&& \displaystyle \sum_{i=1}^k \displaystyle \sum_{j=1}^{b_i} (-1)^{i} a_{ij}^{k+t}(c-a_{ij})^{k+t}(c-2a_{ij}) \end{array} \right| $$
So $ \displaystyle \sum_{ j_i } \displaystyle \prod_{i=1}^k \alpha_{ij_i}(c-2a_{ij_i}).V(\alpha_{1j_1} \ldots \alpha_{kj_k}) f_l(d_{1j_1}, \ldots , d_{sj_s}) =  (s+l)!e_ldetL$ and hence the result.\\
\\
Case $2.$ $s=2k$ is even. Now the resolution starts as in (\ref{eq4}) and  by lemmas \ref {lem4.2}, \ref{lem4.6} and remark \ref{rem4.6} we obtain\\
 $\displaystyle \sum_{ j_i } \displaystyle \prod_{i=1}^k \alpha_{ij_i}V(\alpha_{1j_1} \ldots  \alpha_{kj_k}).$
 
\noindent \hspace{0.15cm} $\displaystyle \sum_{0 \leq r \leq l} (-1)^{l-r}\nu_{l-r} \displaystyle \sum _{t = 0}^{[r/2]}(-1)^{t} \left[{k+r-t \choose k+t} +{k+r-t -1\choose k+t-1} \right]c^{r-2t} \displaystyle \sum_{ \sum \beta_i=t} \prod_{i=1}^k \alpha_{ij_i}^{\beta_i}=$\\
 $\displaystyle \sum_{1 \leq j_i \leq b_i}  \displaystyle \prod_{i=1}^k \alpha_{ij_i}V(\alpha_{1j_1}, \ldots , \alpha_{kj_k}) \left( f_l(a_{1j_1}, \ldots a_{kj_k},\ldots , c) +f_l(a_{1j_1}, \ldots c-a_{kj_k},\ldots , c ) \right)$
 
\noindent Since $f_l(d_{1j_1}, \ldots  d_{sj_s}) \geq 0$ and $V(\alpha_{1j_1}, \ldots , \alpha_{kj_k}) \geq 0$, we have\\
$\displaystyle \sum_{j_i}  \displaystyle \prod_{i=1}^k p_iV(\alpha_{1j_1}, \ldots , \alpha_{kj_k}) (f_l \left( m_1, \ldots  m_k, M_{k+1}, \ldots  M_{s}\right) +f_l(m_1, \ldots m_{k-1},M_k ,\ldots M_s))  $

$ \leq \displaystyle \sum_{1 \leq j_i \leq b_i}  \displaystyle \prod_{i=1}^k \alpha_{ij_i}V(\alpha_{1j_1}, \ldots , \alpha_{kj_k}) \left( f_l(a_{1j_1}, \ldots a_{kj_k},\ldots , c) +f_l(a_{1j_1}, \ldots c-a_{kj_k},\ldots , c ) \right)$

$ \leq \displaystyle \sum_{j_i}   \displaystyle \prod_{i=1}^k P_i V(\alpha_{1j_1}, \ldots , \alpha_{kj_k})( f_l \left( M_1, \ldots  M_k, m_{k+1}, \ldots  m_{s}\right) +f_l(M_1, \ldots M_{k-1},m_k ,\ldots m_s))$

But, $$f_l \left( m_1, \ldots  m_k, M_{k+1}, \ldots  M_{s}\right) \leq f_l(m_1, \ldots m_{k-1},M_k ,\ldots  M_s)$$ and 
$$f_l(M_1, \ldots M_{k-1},m_k ,\ldots  m_s) \leq  f_l \left( M_1, \ldots  M_k, m_{k+1}, \ldots  m_{s}\right) . $$ 
Therefore,

$\displaystyle \sum_{j_i}  \displaystyle \prod_{i=1}^k p_iV(\alpha_{1j_1}, \ldots , \alpha_{kj_k}). 2 f_l \left( m_1, \ldots  m_k, M_{k+1}, \ldots  M_{s}\right)$  

\hspace{0.2cm}$ \leq \displaystyle \sum_{1 \leq j_i \leq b_i}  \displaystyle \prod_{i=1}^k \alpha_{ij_i}V(\alpha_{1j_1}, \ldots , \alpha_{kj_k})  \left( f_l(a_{1j_1}, \ldots a_{kj_k},\ldots , c) +f_l(a_{1j_1}, \ldots c-a_{kj_k},\ldots , c ) \right)$

\hspace{2cm} $ \leq \displaystyle \sum_{j_i}   \displaystyle \prod_{i=1}^k P_i V(\alpha_{1j_1}, \ldots , \alpha_{kj_k}) 2. f_l \left( M_1, \ldots  M_k, m_{k+1}, \ldots  m_{s}\right) $\\
 which is the same as
 
$2f_l\left( m_1, \ldots  m_k, M_{k+1}, \ldots  M_{s}\right) \displaystyle \prod_{i=1}^k p_i det(Q') \leq $

$ \displaystyle \sum_{1 \leq j_i \leq b_i}  \displaystyle \prod_{i=1}^k \alpha_{ij_i}V(\alpha_{1j_1}, \ldots , \alpha_{kj_k}) \left( f_l(a_{1j_1}, \ldots a_{kj_k},\ldots , c) +f_l(a_{1j_1}, \ldots c-a_{kj_k},\ldots , c ) \right)$

\hspace{5cm} $\leq 2f_l\left( M_1, \ldots  M_k, m_{k+1}, \ldots  m_{s}\right) \displaystyle \prod_{i=1}^k P_i det(Q')$\\ where
 $$Q'  =  \left( \begin{array}{cccccc}
r_1 \ldots  &r_i \ldots    & r_k\\
\displaystyle \sum_{j=1}^{r_1}\alpha_{1j} \ldots  &\displaystyle \sum_{j=1}^{r_i}\alpha_{ij} \ldots   & \displaystyle\sum_{j=1}^{r_k}\alpha_{kj}\\
\vdots &\vdots& \vdots& \\
\displaystyle \sum_{j=1}^{r_1}\alpha_{1j}^{k-1} \ldots & \displaystyle \sum_{j=1}^{r_i}\alpha_{ij}^{k-1} \ldots   & \displaystyle\sum_{j=1}^{r_k}\alpha_{kj}^{k-1}\\
\end{array} \right)$$
$det Q' >0$ since at least on of the $ V(\alpha_{1j_1}, \ldots , \alpha_{kj_k}) > 0.$
Replacing the last column by the alternating sums of columns and using theorems \ref{thmS2} and \ref{th2}, we get

$$det Q' = det \left(\begin{array}{cccc}\ldots  &-1 \\
L'& 0\end{array} \right) = detL'$$
Then\\
$2f_l\left( m_1, \ldots  m_k, M_{k+1}, \ldots  M_{s}\right) \displaystyle \prod_{i=1}^k p_i det(L') \leq$

 $\displaystyle \sum_{1 \leq j_i \leq b_i}  \displaystyle \prod_{i=1}^k \alpha_{ij_i}V(\alpha_{1j_1}, \ldots , \alpha_{kj_k})  \left( f_l(a_{1j_1}, \ldots a_{kj_k},\ldots , c) +f_l(a_{1j_1}, \ldots c-a_{kj_k},\ldots , c ) \right)$

\hspace {6cm} $ \leq 2f_l\left( M_1, \ldots  M_k, m_{k+1}, \ldots  m_{s}\right)\displaystyle \prod_{i=1}^k P_i det(L')$\\
On the other hand, we start with $N_{t}$ again. Replacing the last column of $N_{t}$ by alternating sums of the columns and using theorem \ref{th2}, we get:

$$N_{t} = \left|\begin{array}{cccc} L'  &&&0 \\
\ldots  &&& \displaystyle \sum_{i=1}^k \displaystyle \sum_{j=1}^{b_i} (-1)^{i} a_{ij}^{k+t}(c-a_{ij})^{k+t} \end{array} \right| $$\
So $(s+l)!e_ldetL'=$\\ 
$$\displaystyle \sum_{1 \leq j_i \leq b_i}  \displaystyle \prod_{i=1}^k \alpha_{ij_i}V(\alpha_{1j_1}, \ldots , \alpha_{kj_k}) \left( f_l(a_{1j_1}, \ldots a_{kj_k},\ldots , c) +f_l(a_{1j_1}, \ldots c-a_{kj_k},\ldots , c ) \right). $$
\\
We get,\\
$2f_l\left( m_1, \ldots  m_k, M_{k+1}, \ldots  M_{s}\right) \displaystyle \prod_{i=1}^k p_i \leq (s+l)!e_l $

\hspace{7cm} $ \leq   2f_l\left( M_1, \ldots  M_k, m_{k+1}, \ldots  m_{s}\right)\displaystyle \prod_{i=1}^k P_i $\\
\\
\\
$2f_l\left( m_1, \ldots  m_k, M_{k+1}, \ldots  M_{s}\right) m_1\ldots m_kM_kM_{k+1}\ldots M_{2k-1}, \leq (s+l)!e_l $\\

\hspace{3cm} $ \leq   2f_l\left( M_1, \ldots  M_k, m_{k+1}, \ldots  m_{s}\right)M_1 \ldots M_k m_k m_{k+1} \ldots m_{2k-1}$\\

\vspace{0.5cm}

$f_l\left( m_1, \ldots  m_k, M_{k+1}, \ldots  M_{s}\right) m_1\ldots m_kM_{k+1}\ldots M_{2k-1}(2M_k), \leq (s+l)!e_l $\\

\hspace{3cm} $ \leq   f_l\left( M_1, \ldots  M_k, m_{k+1}, \ldots  m_{s}\right)M_1 \ldots M_k m_{k+1} \ldots m_{2k-1}(2m_k)$\\
\\
Now $M_k = c-a_{k1}$ and $m_k=a_{k1}.$\\
Clearly $c \geq 2a_{k1} =2m_k$. Also $2(c-a_{k1})=c+(c-2a_{k1}) \geq c.$\\
So $M_{2k}=c \leq 2M_k$ and $2m_k \leq c=m_{2k}$ and hence\\
\\
 $f_l\left( m_1, \ldots  m_k, M_{k+1}, \ldots  M_{s}\right) m_1\ldots m_kM_{k+1}\ldots M_{2k-1}M_{2k}, \leq (s+l)!e_l $\\

\hspace{3cm} $ \leq   f_l\left( M_1, \ldots  M_k, m_{k+1}, \ldots  m_{s}\right)M_1 \ldots M_k m_{k+1} \ldots m_{2k-1}m_{2k}$ 

This completes the proof.
\end{proof}

Let $R = k[x_1, \ldots ,x_n]$ and $S=R/I$ where  $I$ is a homogeneous ideal of height $s$.

 \begin{corollary}  Let $S=R/I$ be as above. Suppose the betti diagram of $S$ is symmetric, that is $\beta_{ij}= \beta _{s-i,c-j}, 0\le i \le s$ and $\beta_{sj}=0, j\neq c$.  
  Then  the Hilbert coefficients will satisfy the same bounds as in theorem \ref{main}
  \end{corollary}
  
  \begin{corollary} Let $S= R/I$ be as above.  Suppose the betti diagram of $S$ is a positive rational linear combinations of quasi-pure Gorenstein (symmetric) betti diagrams. Then the Hilbert Coefficients of $S$ satisfy the same bounds as in theorem \ref{main}. 
  \end{corollary}
  


\begin{thebibliography}{10} \label{bib}
\bibitem{BS08}{\sc M. Boij and J. S\"oderberg.}
\newblock {\em Graded Betti Numbers of Cohen-Macaulay Modules and the Multiplicity conjecture}, 
\newblock {math. AC/0611081}, (2) \textbf{78}, (2008), (1), 85-106

\bibitem{BH}{\sc W. Bruns and J. Herzog.}
\newblock {\em Cohen-Macaulay rings,}
\newblock {Cambridge University Press}, (1993), 147-205

\bibitem{ES}{\sc D. Eisenbud and F. Schreyer.}
\newblock {\em Betti numbers of graded modules and cohomology of vector bundles},
\newblock {J. Amer. Math. Soc.}, \textbf{22}, (2009), 859-888.

\bibitem {E} {\sc S. El Khoury.}
\newblock{\em Upper Bounds for the Hilbert Coefficients of almost Cohen-Macaulay Algebras, }
\newblock Int. J. Algebra, Vol. 5, (2011), no. \textbf {14}, 679 - 690

\bibitem{HS}{\sc J. Herzog and H. Srinivasan.}
\newblock {\em Bounds for multiplicities},
 \newblock{Transactions of the AMS}
 \newblock 350 nbr \textbf{7}, (1998), 2879-2902.

\bibitem{HZ}{\sc J. Herzog and X. Zheng.}
\newblock{\em Bounds for Hilbert Coefficients}, 
\newblock{Proceedings of the AMS}, 37 nbr \textbf{2}, (2009), 487-494.

\bibitem{HM}{\sc C. Huneke and M.MIller.}
\newblock{\em A note on the multiplicity of Cohen-Macaulay algebras with pure resolutions}, 
\newblock{Canad.J. Math.}, 37, (1985), 1149-1162.

\bibitem {S} {\sc H. Srinivasan.}
\newblock{\em A note on the Multiplicities of Gorenstein Algebras},
\newblock J. Algebra
\newblock  \textbf{208}, (1998), 425-443.


\bibitem {PS} {\sc C. Peskine and L. Szpiro.}
\newblock{\em Syzygies and multiplicities},
\newblock C.R. Acad. Sci. Paris. Ser. A
\newblock  \textbf{278}, (1974), 1421-1424.


\end{thebibliography}
\end{document}